\documentclass[11pt]{amsart}
\usepackage{amsmath,amsthm,amsfonts,amssymb}
\usepackage[mathscr]{euscript}
\usepackage{hyperref}
\usepackage{mathrsfs}
\usepackage{graphicx}
\usepackage{color}
\usepackage{epsfig}
\newtheorem{theorem}{Theorem}[section]

\newtheorem{proposition}{Proposition}[section]
\newtheorem{lemma}{Lemma}[section]

\begin{document}

\title[The critical Choquard equation]{High and low perturbations of the critical
Choquard equation on the Heisenberg group}
 
\date{}
\maketitle

\vspace{ -1\baselineskip}

{\small
\begin{center}
 {\sc Shujie Bai} \\
College of Mathematics, Changchun Normal University\\
 Changchun, 130032, P.R. China \\[10pt]
{\sc Du\v{s}an D. Repov\v{s}}\\
Faculty of Education,
Faculty of Mathematics and Physics\\
University of Ljubljana, Ljubljana, 1000, Slovenia\\
Institute of Mathematics, Physics and Mechanics,
 Ljubljana, 1000, Slovenia\\[10pt]
{\sc Yueqiang Song}
\footnote{Corresponding author: songyq16@mails.jlu.edu.cn}\\
College of Mathematics, Changchun Normal University\\
 Changchun, 130032, P.R. China \\[10pt]
 
\end{center}}

\numberwithin{equation}{section}
\allowdisplaybreaks

 \smallskip

 \begin{quote}
\footnotesize
{\bf Abstract.}
We study the
following critical Choquard equation on
the Heisenberg group:
\begin{equation*}
 \begin{cases}
 \displaystyle
{-\Delta_H u }={\mu}
|u|^{q-2}u+\int_{\Omega}
\frac{|u(\eta)|^{Q_{\lambda}^{\ast}}}
{|\eta^{-1}\xi|^{\lambda}}
d\eta|u|^{Q_{\lambda}^{\ast}-2}u
&\mbox{in }\  \Omega, \\
 u=0 &\mbox{on }\ \partial\Omega,
\end{cases}
\end{equation*}
where
$\Omega\subset \mathbb{H}^N$
is a smooth bounded domain,
$\Delta_H$ is the Kohn-Laplacian
on the Heisenberg group
$\mathbb{H}^N$,
$1<q<2$ or
$2<q<Q_\lambda^\ast$,
$\mu>0$,
$0<\lambda<Q=2N+2$, and
$Q_{\lambda}^{\ast}=\frac{2Q-\lambda}{Q-2}$
is the critical exponent.
Using the concentration compactness
principle and the critical point theory,
we prove that the above
problem has the least two positive
solutions for
$1<q<2$ in the case
of low perturbations (small values of
$\mu$), and
has a nontrivial solution for
$2<q<Q_\lambda^\ast$  in the case of
high perturbations (large values of
$\mu$). Moreover, for
$1<q<2$,
we also show that there is a positive
ground state solution, and for
$2<q<Q_\lambda^\ast$, there are at least
$n$
pairs of nontrivial weak solutions.

\medskip

\emph{\bf Keywords:} Choquard equation; Heisenberg group;
Concentration compactness principle; Hardy-Littlewood-Sobolev
critical exponent.

\medskip

\emph{\bf Math. Subj. Classif. (2020):}  35J20, 35R03, 46E35

\end{quote}

\section{Introduction and main results}

In this paper, our aim is to study
the existence of solutions for
the following critical Choquard
equation on the Heisenberg group:
\begin{equation}\label{1.1}
 \begin{cases}
 \displaystyle
{-\Delta_H u }=\mu
|u|^{q-2}u+\int_{\Omega}
\frac{|u(\eta)|^{Q_{\lambda}^{\ast}}}
{|\eta^{-1}\xi|^{\lambda}}
d\eta|u|^{Q_{\lambda}^{\ast}-2}u
&\mbox{in }\  \Omega, \\
 u=0  &\mbox{on }\ \partial\Omega,
\end{cases}
\end{equation}
where
$\Omega\subset \mathbb{H}^N$
is a smooth bounded
domain,
$\Delta_H$ is the Kohn-Laplacian
on the Heisenberg group
$\mathbb{H}^N$,
$1<q<2$ or
$2<q<Q_{\lambda}^{\ast}$,
$\mu>0$,
$0<\lambda<Q
=2N+2$, and
$Q_{\lambda}^{\ast}=\frac{2Q-\lambda}{Q-2}$
is the
critical exponent.

The study of this problem was mainly
inspired by two aspects. On
the one hand, in the Euclidean case,
more and more mathematicians are
beginning to pay attention to the
Choquard equation. As is well known,
Fr\"{o}hlich \cite{HF} and Pekar
\cite{kh} established the following
Choquard equation
$$
-\Delta u+u=
\Big ( \frac{1}{|x|}\ast|u|^{2}
\Big ) u\quad \text{in}\  \mathbb{R}^{3},
$$
for the first time in their pioneering work of the modeling of
quantum polaron. Such problems are often referred to as the
nonlinear Schr\"{o}dinger-Newton equation. Many authors began to
study these problems by using variational methods. For example,
Lions \cite{Lions2} obtained the existence of an infinite number of
radially symmetric solutions in $H^1(\mathbb{R}^N)$. Ackermann
\cite{ack} proved the existence of an infinite number of
geometrically different weak solutions for a general case was
established. Moroz and Van Schaftingen \cite{moroz4, moroz2}
obtained the properties of the ground state solutions, and also
proved that these solutions decay asymptotically at infinity.
Recently, more and more mathematicians have shown a strong interest
in studying critical Choquard type equations. Br\'{e}zis and Lieb
\cite{HB} originally addressed the critical problem in his seminal
paper, which dealt with the Laplacian equations. Liang et al.
\cite{liang2} proved the multiplicity results of the
Choquard-Kirchhoff type equations with Hardy-Littlewood-Sobolev
critical exponents. More results about Choquard equations are
available in \cite{liang3, PXZ,  WX, zh1}.

On the other hand, the study of nonlinear partial differential
equations on the Heisenberg group has brought about widespread
attention of many researchers. One of the reasons to study such
equations is due to their many significant applications. Over the
last few decades, many scholars have paid close attention to
Heisenberg group's geometric analysis because of its significant
applications in quantum mechanics, partial differential equations
and other fields. For example, Liang and Pucci \cite{Liang5} applied
the Symmetric Mountain Pass Theorem to considering a class of the
the critical Kirchhoff-Poisson systems on the Heisenberg group.
Pucci and Temperini \cite{pu2} proved the existence of entire
nontrivial solutions for the $(p, q)$ critical systems on the
Heisenberg group by an application of variational methods. Pucci
\cite{pu4}  applied the Mountain Pass Theorem and the Ekeland
variational principle to prove the existence of nontrivial
nonnegative solutions of the Schr\"{o}dinger-Hardy system on the
Heisenberg group. For more fascinating results, see \cite{an, bor,
Liu2, Liu,  pu3, pu4, pu5, pu6}. However, once we turn our attention
to the critical Choquard equation  on the Heisenberg group, we
immediately notice that the literature is relatively sparse.
Recently, Goel and Sreenadh \cite{go} have studied the following
critical Choquard equation on the Heisenberg group:
\begin{equation*}
\left\{
\begin{array}{lll}
-\Delta_Hu=au+
\Big (
\displaystyle
\int_\Omega
\frac{|u(\eta)|^{Q_\lambda^\ast}}
{|\eta^{-1}\xi|^\lambda}
\Big )
|u|^{Q_\lambda^\ast-2}u
&\mbox{in}\  \Omega, \\
u=0
&\mbox{on}\ \partial\Omega.
\end{array} \right.\end{equation*}
They applied the boot-strap method,
iteration techniques, the
linking theorem, and the Mountain
Pass Theorem to obtain the
regularity of solutions and
nonexistence of solutions for this kind
of problems.

Sun et al. \cite{sun} studied the following critical
Choquard-Kirchhoff problem on the Heisenberg group:
$$
M(\|u\|^2)(-\Delta_{H}u+u)=
\int_{\mathbb{H}^N}
\frac{|u(\eta)|^{Q_\lambda^\ast}}
{|\eta^{-1}\xi|^\lambda}
d\eta|u|^{Q_\lambda^\ast-2}u+\mu f(\xi,u),
$$
where
$f$ is a Carath\'{e}odory function,
$M$ is the Kirchhoff
function,
$\mu>0$ is a parameter, and
$Q_\lambda^\ast=\frac{2Q-\lambda}{Q-2}$
is the critical exponent  in
the sense of Hardy-Littlewood-Sobolev
inequality. A new version of
the concentration-compactness
principle of the Choquard equation on
the Heisenberg group was established.
Moreover, they also applied
the Mountain Pass Theorem to
obtain the existence of nontrivial
solutions for the above-mentioned
problem under non-degenerate and
degenerate conditions.

Inspired by the above achievements,
with the help of the
concentration compact principle and
the critical point theory, we
prove that problem \eqref{1.1}
has at least two positive
solutions for
$1<q<2$ and
$\mu$ small enough, and this equation has
a nontrivial solution for
$2<q<Q_\lambda^\ast$ and
$\mu$ large
enough. Moreover, for
$1<q<2$, we also show that there is a positive
ground state solution for problem \eqref{1.1},
and for
$2<q<Q_\lambda^\ast$,  there are at least
$n$ pairs of nontrivial
weak solutions.

Before presenting the main results of
this paper, we list some
notions about the Heisenberg group.
 Let
$\mathbb{H}^N$ be the
Heisenberg group. If
$\xi=(x,y,t)\in\mathbb{H}^N$, then the
definition of this group operation is
$$
 \tau_{\xi}(\xi')=\xi\circ\xi'=
 (x+x',y+y',t+t'+2(x'y-y'x))
 \ \ \mbox{for all }\ \xi,\xi'\in \mathbb{H}^N.
 $$
$\xi^{-1}=-\xi$ is the inverse, and therefore
$(\xi\circ
\xi')^{-1} = (\xi')^{-1}\circ {\xi}^{-1}$.

The definition of a natural group of dilations on
$\mathbb{H}^N$ is
$\delta_s(\xi)=(sx,sy,s^2t)$ for all
$s>0$. Hence,
$\delta_s(\xi_0\circ\xi) =
\delta_s(\xi_0)\circ \delta_s(\xi)$. It
can be easily proved that the
Jacobian determinant of dilatations
$\delta_s: \mathbb{H}^N\to \mathbb{H}^N$
is constant and equal to
$s^Q$ for all
$\xi=(x,y,t)\in \mathbb{H}^N$.
The natural number
$Q=2N+2$ is called the homogeneous
dimension of
$\mathbb{H}^N$ and
the critical exponents is
$Q^{*}:= \frac{2Q}{Q-2}$.
The Kor\'{a}nyi
norm is defined as follows
$$|\xi|_{H}=
\left[(x^2+y^2)^2+t^2\right]^{\frac{1}{4}}\
 \ \mbox{for all}\ \ \xi\in \mathbb{H}^N,
$$
and is derived from an anisotropic
dilation on the Heisenberg group.
Hence, the homogeneous degree of the
Kor\'{a}nyi norm is equal to 1, in terms of dilations
$$
\delta_s:(x,y,t)\mapsto(sx,sy,s^2t)\ \
 \mbox{for all }\ s>0.
 $$
The set
$$
B_{H}(\xi_{0},r)=\{\xi \in
\mathbb{H}^{N}:d_{H}(\xi_0,\xi)<r\}
$$
 denotes the  Kor\'{a}nyi open
ball of radius r centered at
$\xi_{0}$. For the sake of simplicity,
we denote
$B_{r} = B_{r}(O)$, where
$O = (0,0)$ is the natural
origin of
$\mathbb{H}^{N}$.

The following vector fields
$$
X_{j}=\frac{\partial}{\partial
x_{j}}+2y_{j}\frac{\partial}{\partial t},
\quad
Y_{j}=\frac{\partial}{\partial y_{j}}-2x_{j}
\frac{\partial}{\partial
t},\quad T=\frac{\partial}{\partial t},
$$
generate the real Lie algebra of left
invariant vector fields for
$j=1,\cdots,n$,
which forms a basis satisfying
the Heisenberg regular commutation relation on
$\mathbb{H}^{N}$.
This means that
$$
[X_{j}, Y_{j}]=-4\delta_{jk}T,
\quad
[Y_{j},Y_{k}]=[X_{j},X_{k}]=[Y_{j},T]=
[X_{j},T]=0.
$$
The so-called horizontal
vector field is just a vector field with
the span of
$[X_{j}, Y_{j}]^{n}_{j=1}$.
The Heisenberg gradient on
$\mathbb{H}^{N}$ is
$$
\nabla_H
=(X_{1},X_{2},\cdots,X_{n},Y_{1},Y_{2},
\cdots,Y_{n}),
$$
and the Kohn
Laplacian on
$\mathbb{H}^{N}$
is given by
\begin{align*}
\Delta_{H} & =\sum_{j=1}^{N}X_{j}^{2}+Y_{j}^{2}
\\
&
=
\sum_{j=1}^{N}
 \Big [\frac{\partial^{2}}
{\partial x_{j}^{2}}+\frac{\partial^{2}}
{\partial y_{j}^{2}}+4y_{j}
\frac{\partial^{2}}
{\partial x_{j}\partial t}-4x_{j}
\frac{\partial^{2}}
{\partial x_{j}\partial t}+
4(x_{j}^{2}+y_{j}^{2})\frac{\partial^{2}}
{\partial t^{2}}
\Big ].
\end{align*}

The Haar measure is invariant under the left translations of the
Heisenberg group and is $Q$-homogeneous in terms of dilations. More
precisely, it is consistent with the $(2n + 1)$-dimensional Lebesgue
measure. Hence, as is shown in Leonardi and Masnou \cite{Leonardi},
the topological dimension $2N + 1$ of $\mathbb{H}^{N}$ is strictly
less than its Hausdorff dimension $Q = 2N + 2$. Next, $|\Omega|$
denotes the $(2N + 1)$ dimensional Lebesgue measure of any
measurable set $\Omega\subseteq \mathbb{H}^{N}$. Therefore,
$$
|\delta_{s}(\Omega)|=s^{Q}|\Omega|,
\quad d(\delta_{s}\xi)=s^{Q}d\xi
$$
and
$$
|B_{H}(\xi_{0},r)|
=\alpha_{Q}r^{Q},\quad
\mbox{where}\quad
 \alpha_{Q}=|B_{H}(0,1)|.
$$

Now, we are ready to present our main results.

\begin{theorem}\label{the1.1}
Let
$\Omega\subset \mathbb{H}^{N}$ be
a smooth bounded domain and
$1<q<2$. Then there exists
$\mu_{\ast}>0$ such that if
$\mu\in(0,
\mu_{\ast})$, then problem \eqref{1.1}
has at least two positive
solutions. Moreover, problem \eqref{1.1}
has a positive ground state
solution.
\end{theorem}

\begin{theorem}\label{the1.2} Let
$\Omega\subset \mathbb{H}^{N}$ be a
smooth bounded domain and
$2<q<Q_\lambda^\ast$. Then there exists
$\mu^{\ast}>0$ such that if
$\mu > \mu^{\ast}$, then problem \eqref{1.1}
has a nontrivial solution.
\end{theorem}

\begin{theorem}\label{the1.3}
Let
$\Omega\subset \mathbb{H}^{N}$
be a smooth bounded domain and
$2<q<Q_\lambda^\ast$. Then there exists
$\mu^{\ast\ast}>0$ such that
if
$\mu > \mu^{\ast\ast}$,
then problem \eqref{1.1} has at least
$n$
pairs of nontrivial weak solutions.
\end{theorem}

The paper is organized as follows.
In Section 2, we collect
some notations and known facts,
and introduce some properties
of the Folland-Stein space
$\mathring{S}_1^2(\Omega)$.
Moreover,
a key estimate, i.e.,
 Lemma \ref{lem2.2}, is introduced.
 In Section 3, we make use of the
 variational methods to prove some
basic lemmas.
Then we demonstrate Theorem 1.1.
To be more specific,
in the first subsection,
Ekeland variational principle
is used to prove the existence of the
first positive solution, and in the
second subsection,
Mountain Pass Lemma is used to prove
the existence of the second positive
solution. Furthermore,
in the third subsection, we prove
that problem \eqref{1.1} has a positive
ground state solution.
In Section 4, we use the general Mountain
Pass Theorem to accomplish the
proof of Theorem 1.2.
Finally, in Section 5,
we prove Theorem 1.3 by using
Krasnoselskii's genus theory.

\section{Preliminaries}

In this section, we have collected
some known facts which will be
useful in the sequel (for others we refer to Papageorgiou et al. \cite{PRR}). Set
$Q=2N+2$ and
$Q^{\ast}=\frac{2Q}{Q-2}$.
Let
$\|u\|_p^{p}=\int_{\Omega}
|u|^pd\xi$ for all
$u\in
L^{p}(\Omega),$  represent the usual
$L^p$-norm. Following Folland
and Stein \cite{fo}, we define the space
$\mathring{S}_1^2(\Omega)$
as the closure of
$C_{0}^{\infty}(\Omega)$ in
$S_1^2(\mathbb{H}^{N})$. Then
$\mathring{S}_1^2(\Omega)$ is a
Hilbert space with respect to the norm
$$
\|u\|_{\mathring{S}_1^2(\Omega)}^2=
\int_{\Omega}|\nabla_{H}u|^2d\xi.
$$
For the sake of brevity,
we shall denote
$\|u\|=\|u\|_{\mathring{S}_1^2(\Omega)}^2$.
By
Folland
and Stein
 \cite{fo}, we know
that the Folland-Stein space is a
Hilbert space and the embedding
$\mathring{S}_1^2(\Omega)\hookrightarrow
L^p(\Omega)
$ for all
$p
\in [1, Q^{*})$ is compact. However,
it is only continuous if
$p =
Q^{*}$.

By Jerison and Lee \cite{jer},
we have the following Best Sobolev
constant
\begin{equation}\label{e2.3}
S=\inf\limits_{u\in
\mathring{S}_1^2(\Omega)}
\frac{ \int_{\Omega}|\nabla_{H}u|^2d\xi}
{ \big ( \int_{\Omega}
|u|^{Q^*}d\xi
 \big ) ^{\frac{2}{Q^*}}}.
\end{equation}

\begin{proposition}\label{pro2.1}
{\rm (see  Goel and Sreenadh \cite{go})} Let $r, s >1$ and
$0<\lambda<Q$ with $\frac{1}{r}+\frac{\lambda}{Q}+\frac{1}{s}=2$, $g
\in L^{r}(\Omega)$, and $d \in L^{s}(\Omega)$. There is a sharp
constant $C(t,r,\lambda,Q)$, independent of g, d, such that
\begin{equation}\label{2.5}
\int_{\Omega}\int_{\Omega}
\frac{g(\xi)d(\eta)}
{|\eta^{-1}\xi|^{\lambda}}d\eta d\xi\leq
C(t,r,\lambda,Q)|g|_{r}|d|_{s}.
\end{equation}
If
$r=s=\frac{2Q}{2Q-\lambda}$, then
$$
C(t,r,\lambda,Q)=C(\lambda,Q)=
\Big ( \frac{\pi^{N+1}}{2^{N-1}N!}
\Big ) ^{\lambda/Q}\frac{N!\Gamma((Q-\lambda)/2)}
{\Gamma^{2}((2Q-\lambda)/2)},
$$
where $\Gamma$ is the standard Gamma function.
\end{proposition}

From Goel and Sreenadh \cite{go},
we get
$$
\int_{\Omega}\int_{\Omega}
\frac{|u(\xi)|^{Q_{\lambda}^{\ast}}
|u(\eta)|^{Q_{\lambda}^{\ast}}}{|\eta^{-1}\xi
|^{\lambda}}d\eta d\xi\leq
C(\lambda,Q)|u|_{Q^{\ast}}^{2Q_{\lambda}^{\ast}}
$$
and the best constant
$S_{HG}$ is defined by
\begin{equation}
\label{2.7}
S_{HG}=\inf\limits_{u\in
\mathring{S}_1^2(\Omega)
\setminus\{0\}}
\frac{\int_{\Omega}|\nabla_{H}u|^2d\xi}
{ \big (\int_{\Omega}
\int_{\Omega}\frac{|u(\xi)|^{Q_{\lambda}^{\ast}}
|u(\eta)|^{Q_{\lambda}^{\ast}}}
{|\eta^{-1}\xi|^{\lambda}}d\eta
d\xi
\big ) ^{\frac{1}{Q_{\lambda}^{\ast}}}}.
\end{equation}

\begin{lemma}\label{lem2.1}
{\rm (see  Goel and Sreenadh \cite{go})} We obtain the best constant
$S_{HG}$ if and only if
$$
u(\xi) = u(x,y,t) = CZ(\delta_{\theta}(a^{-1}\xi)),
$$
where
$C>0$ is a fixed constant,
$\theta\in(0, \infty)$ are
parameters,
$a\in \mathbb{H}^{N}$ and Z is defined in
Goel and  Sreenadh
\cite[(1.6)]{go}. Furthermore,
$$
S_{HG} = S(C(Q,\lambda))^{\frac{-1}
{Q_{\lambda}^{\ast}}},
$$
where S is the best constant
defined in 
Goel and  Sreenadh
\cite[(1.5)]{go}.
\end{lemma}

On the other hand, from the proof of
Goel and  Sreenadh
\cite[Lemma 2.1]{go},
we know that a unique minimizer of
$S_{HG}$ is the function
$$
P(\eta)=S^{\frac{(Q-\lambda)(2-Q)}{4(Q-\lambda+2)}}
C(Q,\lambda)^{\frac{2-Q}{2(Q-\lambda+2)}}Z(\eta)
$$
and it satisfies
the following:
\begin{equation*}
-\Delta_{H}u=\Big(\int_{\mathbb{H}^N}
\frac{|u(\eta)|^{Q_\lambda^\ast}}
{|\eta^{-1}\xi|^\lambda}{d\eta}\Big)
|u|^{Q_\lambda^\ast-2}u  \ \
\mbox{in} \ \ \mathbb{H}^N
\end{equation*}
and
$$
\int_{\mathbb{H}^N}|\nabla_{H} P|^2d\xi
=\int_{\mathbb{H}^N}\int_{\mathbb{H}^N}
\frac{|P(\xi)|^{Q_\lambda^\ast}
|P(\eta)|^{Q_\lambda^\ast}}
{|\eta^{-1}\xi|^\lambda}d\eta
d\xi =S_{HG}^{\frac{2Q-\lambda}{Q-\lambda+2}}.
$$
Furthermore, for
$\gamma>0$, the function
$P_{\gamma}$ is defined as follows
$$
P_{\gamma}=
\frac{\gamma^{\frac{Q-2}{2}}
S^{\frac{(Q-\lambda)(2-Q)}
{4(Q-2+\lambda)}}C(\lambda,Q)^{\frac{2-Q}
{2(Q-\lambda+2)}}C}{(\gamma^{4}t^{2}+(1+
\gamma^{2}|x|^{2}+\gamma^{2}
|y|^{2})^{2})^{(Q-2)/4}},
$$
and satisfied
\begin{align*}
\int_{\Omega}|\nabla_H P_{\gamma}|^2d\xi
=\int_{\Omega}\int_{\Omega}
\frac{|P_{\gamma}(\xi)|^{Q_{\lambda}^{\ast}}
|P_{\gamma}(\eta)|^{Q_{\lambda}^{\ast}}}
{|\eta^{-1}\xi|^{\lambda}}d\eta
d\xi=S_{HG}^{\frac{2Q-\lambda}{Q-\lambda+2}}
\end{align*}
and
\begin{align*}
\int_{\Omega}|P_{\gamma}|^{Q^{\ast}}d\xi
=
S^{\frac{Q}{Q-\lambda+2}}
C(\lambda,Q)^{\frac{-Q}{Q-\lambda+2}}.
\end{align*}
More generally, we can suppose that
$0\in\Omega$ and that there is
$r>0$ such that
$B(0,4r)\subset\Omega\subset B(0,kr)$
for some
$k>0$.
Choose
$\nu\in C_{c}^{\infty}(\Omega)$
such that
$0\leq\nu\leq1$,
$|\nabla_{H}\nu|$ is bounded, and
\begin{equation}
\nu(\eta)=
\left\{
\begin{aligned}\label{2.8}
1,&\quad \hbox {if}\quad \eta\in B(0,r),
\\
0,&\quad \hbox {if}\quad \eta\in
\Omega\backslash B(0,2r).
\end{aligned}
\right.
\end{equation}
Then for the following function
\begin{equation}\label{2.9}
\upsilon_{\gamma}=\nu P_{\gamma}\in
 \mathring{S}_1^2(\Omega),
\end{equation}
we have asymptotic estimates as follows.

\begin{lemma}\label{lem2.2}
{\rm (see Goel and Sreenadh \cite{go})} Let $0<\lambda<Q$. Then the
following holds:
\begin{itemize}
\item[($i$)]
\begin{equation*}
\int_{\Omega}|\upsilon_{\gamma}|^{2}d\xi\geq C
\left\{
\begin{aligned}
\gamma^{-2}+O(\gamma^{-Q+2}),&\quad  Q>4,\\
\gamma^{-2}\log\gamma+O(\gamma^{-2}),&\quad Q=4.\\
\end{aligned}
\right.
\end{equation*}
\item[($ii$)]
\begin{align*}
\int_{\Omega}|\upsilon_{\gamma}|^{Q^{\ast}}
d\xi=S^{\frac{Q}{Q-\lambda+2}}
C(\lambda,Q)^{\frac{-Q}
{Q-\lambda+2}}+O(\gamma^{-Q}).
\end{align*}
\item[($iii$)]
\begin{align*}
\int_{\Omega}\int_{\Omega}
\frac{|\upsilon_{\gamma}
(\xi)|^{Q_{\lambda}^{\ast}}
|\upsilon_{\gamma}(\eta)
|^{Q_{\lambda}^{\ast}}}{|\eta^{-1}
\xi|^{\lambda}}d\eta d\xi
\leq
S_{HG}^{\frac{2Q-\lambda}
{Q-\lambda+2}}+O(\gamma^{-Q}).
\end{align*}
\item[($iv$)]
\begin{align*}
\int_{\Omega}\int_{\Omega}
\frac{|\upsilon_{\gamma}
(\xi)|^{Q_{\lambda}^{\ast}}
|\upsilon_{\gamma}
(\eta)|^{Q_{\lambda}^{\ast}}}
{|\eta^{-1}\xi|^{\lambda}}d\eta
d\xi\geq S_{HG}^{\frac{2Q-\lambda}
{Q-\lambda+2}}-
O(\gamma^{-\frac{2Q-\lambda}{2}}).
\end{align*}
\item[($v$)]
\begin{align*}
\int_{\Omega}|\nabla_{H}
\upsilon_{\gamma}|^{2}d\xi\leq
S_{HG}^{\frac{2Q-\lambda}
{Q-\lambda+2}}+
O(\gamma^{-\min\{\frac{2Q-\lambda}{2},Q-2\}}).
\end{align*}
\end{itemize}
\end{lemma}

\section{Low perturbations of problem \eqref{1.1}}

We say that
$u\in \mathring{S}_1^2(\Omega)$
is a solution of problem
\eqref{1.1} if
$$
\int_{\Omega}\nabla_H u\nabla_H vd\xi
-\mu\int_{\Omega}
|u|^{q-2}uvd\xi-
\int_{\Omega}
\int_{\Omega}
\frac{|u(\eta)|^{Q_{\lambda}^{\ast}}}
{|\eta^{-1}\xi|^{\lambda}}
|u|^{Q_{\lambda}^{\ast}-2}uvd\eta d\xi=0
$$
for any
$v\in \mathring{S}_1^2(\Omega)$.
Furthermore, if
$u>0$, then
we call
$u\in \mathring{S}_1^2(\Omega)$
a positive solution to
problem \eqref{1.1}.
 In order to prove our results,
 it is necessary to
define the energy functional
$I_{\mu}: \mathring{S}_1^2(\Omega)
\rightarrow \mathbb{R}$ related
to problem \eqref{1.1}:
\begin{equation}\label{5.1}
I_{\mu}(u)=\frac{1}{2}
\int_{\Omega}|\nabla_H u|^{2}d\xi
-\frac{\mu}{q}\int_{\Omega}|u|^{q}d\xi -
\frac{1}{2Q_{\lambda}^{\ast}}
\int_{\Omega}\int_{\Omega}
\frac{|u(\xi)|^{Q_{\lambda}^{\ast}}
|u(\eta)|^{Q_{\lambda}^{\ast}}}
{|\eta^{-1}\xi|^{\lambda}}d\eta
d\xi.
\end{equation}
Then
$I_{\mu}$ is
$C^{1}$ on
$\mathring{S}_1^2(\Omega)$ and its
critical points are solutions of
problem \eqref{1.1}. Indeed,
let
$I_{\mu}'(u)$ denote the derivative of
$I_{\mu}$ at
$u$, that
is, for any
$u\in \mathring{S}_1^2(\Omega)$,
\begin{align*}
\langle I_{\mu}'(u), v\rangle
& =  \int_{\Omega}\nabla_H u\nabla_H vd\xi
-\mu\int_{\Omega}|u|^{q-2}uvd\xi\\
&
-\int_{\Omega}\int_{\Omega}
\frac{|u(\xi)|^{Q_{\lambda}^{\ast}}
|u(\eta)|^{Q_{\lambda}^{\ast}-2}
u(\eta)v(\eta)}{|\eta^{-1}\xi|^{\lambda}}
d\eta d\xi.
\end{align*}
Then $I_{\mu}'(u)$ continuously maps $\mathring{S}_1^2(\Omega)$ in
the dual space of $\mathring{S}_1^2(\Omega)$, which can be shown by
standard calculations. Therefore, we conclude that $u$ is a solution
of problem \eqref{1.1} if and only if $I_{\mu}$ is $C^{1}$ on
$\mathring{S}_1^2(\Omega)$ and $I_{\mu}'(u)=0.$

\subsection {The existence of a positive
solution of problem \eqref{1.1}}

\begin{lemma}\label{lem3.1}
Let
$1<q<2$. Then for all
\begin{equation}\label{3.1}
c<
\Big ( \frac{1}{2}-\frac{1}{2Q_{\lambda}^{\ast}}
\Big ) S_{HG}^{\frac{Q_{\lambda}^{\ast}}
{Q_{\lambda}^{\ast}-1}}-D\mu^{\frac{2}{2-q}},
\end{equation}
where
$$D=
\Big ( \frac{2Q_{\lambda}^{\ast}-q}
{2Q_{\lambda}^{\ast}q}
|\Omega|^{\frac{Q_{\lambda}^{\ast}-q}
{Q_{\lambda}^{\ast}}}S^{-\frac{q}{2}}
\Big ) ^{\frac{2}{2-q}}
\Big ( \frac{qQ_{\lambda}^{\ast}}
{Q_{\lambda}^{\ast}-1}
\Big ) ^{\frac{q}{2-q}},$$
$I_{\mu}$ satisfies the
$(PS)_{c}$
condition.
\end{lemma}

\begin{proof}[\bf Proof.]
Suppose that
$\{u_{n}\}\subset \mathring{S}_1^2(\Omega)$
satisfies
\begin{equation}\label{3.2}
I_{\mu}(u_{n})\rightarrow c,\ \
 I_{\mu}'(u_{n})\rightarrow 0
\quad as \quad n\rightarrow\infty,
\end{equation}
where
$c$ is taken from \eqref{3.1}.
It follows from the Young
inequality that
\begin{align}\label{3.3}
1+c+o(\|u_{n}\|)
&\geq I_{\mu}(u_{n})-
\frac{1}
{2Q_{\lambda}^{\ast}}I_{\mu}'(u_{n})u_{n}\\
&=\nonumber
\Big ( \frac{1}{2}-\frac{1}{2Q_{\lambda}^{\ast}}
\Big ) \|u_{n}\|^{2}-\mu
\Big ( \frac{1}{q}-\frac{1}{2Q_{\lambda}^{\ast}}
\Big ) \int_{\Omega}|u_{n}|^{q}d\xi\\
&\geq\nonumber
\Big ( \frac{1}{2}-\frac{1}{2Q_{\lambda}^{\ast}}
\Big ) \|u_{n}\|^{2}-\mu
\Big ( \frac{1}{q}-\frac{1}{2Q_{\lambda}^{\ast}}
\Big ) S^{-\frac{q}{2}}
|\Omega|^{\frac{Q_{\lambda}^{\ast}-q}
{Q_{\lambda}^{\ast}}}\|u_{n}\|^{q}.
\end{align}
This means that
$\{u_{n}\}$ is bounded in
$\mathring{S}_1^2(\Omega)$
since
$1<q<2$. More generally, let us assume that
$u_{n}
\rightharpoonup u$ weakly in
$\mathring{S}_1^2(\Omega)$ and
$u_{n}\rightarrow u$ strongly in
$L^{p}(\Omega)$ with
$1\leq
p<Q^{*}$.
Applying the concentration compactness
 principle on the
Heisenberg group (see Sun et al. \cite[Theorem 3.1]{sun}), one has
\begin{align*}
 |u_{n}|^{Q^{*}}
 & \rightharpoonup\zeta\geq
|u|^{Q^{*}}+\sum_{j\in J}\zeta_{j}
\delta_{\xi_{j}}, \\
 \|\nabla_{H}u_{n}|^{2}
& \rightharpoonup d\mu\geq|\nabla_{H}u|^{2}
+\sum_{j\in J}\mu_{j}\delta_{\xi_{j}},\\
\Big ( \displaystyle\int_{\Omega}
\frac{|u_{n}(\eta)|^{Q_{\lambda}^{\ast}}}
{|\eta^{-1}\xi|^{\lambda}}d\eta
\Big ) |u_{n}(\xi)|^{Q_{\lambda}^{\ast}}
& \rightharpoonup
\Big ( \displaystyle\int_{\Omega}
\frac{|u(\eta)|^{Q_{\lambda}^{\ast}}}
{|\eta^{-1}\xi|^{\lambda}}d\eta
 \Big ) |u(\xi)|^{Q_{\lambda}^{\ast}}
 +\sum_{j\in
J}\nu_{j}\delta_{\xi_{j}} ,
\end{align*}
where
$J$ is at most countable index set,
$\xi_{j}\in\Omega$ and
$\delta_{\xi_{j}}$ is the Dirac mass at
$\xi_{j}$. Furthermore, we
have
\begin{equation}\label{3.5}
\zeta_{j}, \mu_{j}, \nu_{j}>0, \quad
S_{HG}\nu_{j}^{\frac{1}{Q_{\lambda}^{\ast}}}
\leq\mu_{j},\quad
\nu_j^{\frac{Q}{2Q-\lambda}} \leq
C(Q,\lambda)^{\frac{Q}{2Q-\lambda}}\zeta_j.
\end{equation}
Now, we claim that $J=\emptyset$. In fact, let us assume that the
hypothesis $\mu_{j}\neq0$ holds for some $j\in J$. Then for
$\varepsilon>0$ small enough, by  Capogna et al.
\cite[Lemma 3.2]{chen}, we can take the cut-off function
$\psi_{\varepsilon,j}\in C_{0}^{\infty}(B_{H}(\xi_{j},\varepsilon))$
 such that
$0\leq\psi_{\varepsilon,j}\leq1$  and
\begin{equation}\label{3.6}
\begin{cases}
\psi_{\varepsilon,j}=1
&\mbox{in} \quad B_{H}(\xi_{j},
\frac{\varepsilon}{2}),\\
\psi_{\varepsilon,j}=0
&\mbox{in}
\quad
\Omega\backslash B_{H}(\xi_{j},
\varepsilon),\\
|\nabla_{H}\psi_{\varepsilon,j}
|\leq\frac{4}{\varepsilon}.
\end{cases}
\end{equation}

Now, by the boundedness of
$\{\psi_{\varepsilon,j}u_{n}\}$ and
\eqref{3.2}, we have
\begin{align*}
\begin{split}
\lim_{n \to\infty}I_{\mu}'(u_{n})
[\psi_{\varepsilon,j}u_{n}]
& = \lim_{n \to\infty}
\Big ( \int_{\Omega}\nabla_H u_{n}\nabla_H
\psi_{\varepsilon,j}u_{n}d\xi-\mu
\int_{\Omega}u_{n}^{q-1}
\psi_{\varepsilon,j}u_{n}d\xi
\\
& -\int_{\Omega}\int_{\Omega}
\frac{|u_{n}|^{Q_{\lambda}^{\ast}}
|u_{n}(\eta)|^{Q_{\lambda}^{\ast}-1}
\psi_{\varepsilon,j}u_{n}(\eta)}
{|\eta^{-1}\xi|^{\lambda}}
d\eta d\xi
\Big ) =0,\\
\end{split}
\end{align*}
which gives
\begin{align}\label{3.7}
&\int_{\Omega}|\nabla_H u_{n}|^{2}
\psi_{\varepsilon,j}d\xi+
\int_{\Omega}\nabla_H u_{n}\nabla_{H}
\psi_{\varepsilon,j}u_{n}d\xi\\
&=\nonumber\mu\int_{\Omega}
u_{n}^{q-1}
\psi_{\varepsilon,j}u_{n}d\xi+
\int_{\Omega}\int_{\Omega}
\frac{|u_{n}(\xi)|^{Q_{\lambda}^{\ast}}
|u_{n}(\eta)|^{Q_{\lambda}^{\ast}-1}
\psi_{\varepsilon,j}u_{n}(\eta)}
{|\eta^{-1}\xi|^{\lambda}}
d\eta d\xi+o(1),
\end{align}
where
$o(1)\rightarrow0$ as
$n\rightarrow\infty$. From
\eqref{3.6}, we   obtain
\begin{align*}
\int_{\Omega}u_{n}^{q-1}
\psi_{\varepsilon,j}u_{n}d\xi
&
\leq|B_{H}(\xi_{j},
\varepsilon)|^{\frac{Q_{\lambda}^{\ast}-q}
{Q_{\lambda}^{\ast}}}
\Big ( \int_{B_{H}(\xi_{j},\varepsilon)}
|u_{n}|^{Q_{\lambda}^{\ast}}d\xi
 \Big ) ^{\frac{q}{Q_{\lambda}^{\ast}}}\\
&\leq|\alpha_{Q}
\varepsilon^{Q}|^{\frac{Q_{\lambda}^{\ast}-q}
{Q_{\lambda}^{\ast}}}S^{-\frac{q}{2}}\|u_{n}\|^{q}.
 \end{align*}
Thus, by the boundedness of
$\{u_{n}\}$, we obtain that
\begin{equation}\label{3.8}
\lim_{\varepsilon \to 0}\lim_{n
\to\infty}\int_{\Omega}u_{n}^{q-1}
\psi_{\varepsilon,j}u_{n}d\xi=0.
\end{equation}
Furthermore, by \eqref{3.7} and the H\"{o}lder inequality, we have
\begin{align}
&\nonumber\lim_{\varepsilon \to 0}\lim_{n
\to\infty}
\Big |\int_{\Omega}u_{n}\nabla_H
u_{n}\nabla_{H}\psi_{\varepsilon,j}d\xi
\Big|\\
&\leq C\lim_{\varepsilon \to
0}
\Big ( \int_{B_{H}(\xi_{j},\varepsilon)}
|u_n|^{Q_{\lambda}^{\ast}}d\xi
\Big ) ^{\frac{1}{Q_{\lambda}^{\ast}}}
\Big ( \int_{B_{H}(\xi_{j},\varepsilon)}
|\nabla_{H}
\psi_{\varepsilon,j}|^{Q_{\lambda}^{\ast}}d\xi
 \Big ) ^{\frac{1}{Q_{\lambda}^{\ast}}}=0.
\end{align}
Hence,
\begin{align}
\label{3.10}
&
\lim_{\varepsilon \to 0}
\lim_{n \to\infty}\int_{\Omega}|\nabla_H
u_{n}|^{2}\psi_{\varepsilon,j}
(\xi)d\xi \\
\notag
\geq
& \lim_{\varepsilon \to
0}
\Big ( \mu_{j}+\int_{B_{H}
(\xi_{j},\varepsilon)}|\nabla_H
u|^{2}\psi_{\varepsilon,j}(\xi)d\xi
\Big ) =\mu_{j}
\end{align}
and
\begin{align}
 \label{3.11}
&  \lim_{\varepsilon \to 0}
\lim_{n
\to\infty}\int_{\Omega}
\int_{\Omega}\frac{|u_{n}(\xi)
|^{Q_{\lambda}^{\ast}}|u_{n}
(\eta)|^{Q_{\lambda}^{\ast}}
\psi_{\varepsilon,j}(\xi)}
{|\eta^{-1}\xi|^{\lambda}}
d\eta d\xi \\
\notag
&
= \lim_{\varepsilon \to
0}
\Big ( \nu_{j}+\int_{B_{H}
(\xi_{j},\varepsilon)}
\int_{\Omega}
\frac{|u(\xi)|^{Q_{\lambda}^{\ast}}
|u(\eta)|^{Q_{\lambda}^{\ast}}
\psi_{\varepsilon,j}(\xi)}
{|\eta^{-1}\xi|^{\lambda}}
d\eta d\xi
\Big ) =\nu_{j}.
\end{align}
So, from \eqref{3.7}-\eqref{3.11},
we   conclude that
$\nu_{j}\geq\mu_{j}$.
Hence, it follows from
\eqref{3.5} that
$\mu_{j}\geq
S_{HG}^{\frac{Q_{\lambda}^{\ast}}
{Q_{\lambda}^{\ast}-1}}$.

Furthermore, according to \eqref{3.2}
and the Young inequality, we
obtain
\begin{align}
\label{3.12}
c
&
=\nonumber\lim_{n \to\infty}
\Big\{I_{\mu}(u_{n})-\frac{1}
{2Q_{\lambda}^{\ast}}I_{\mu}'(u_{n})u_{n}
\Big\}\\
&
=\lim_{n
\to\infty}
\Big\{(\frac{1}{2}-\frac{1}
{2Q_{\lambda}^{\ast}})\|u_{n}\|^{2}-
\mu(\frac{1}{q}-\frac{1}
{2Q_{\lambda}^{\ast}})
\int_{\Omega}|u_{n}|^{q}d\xi
\Big\}\\
&
\nonumber
\geq(\frac{1}{2}-\frac{1}
{2Q_{\lambda}^{\ast}})\mu_{j}+
(\frac{1}{2}-\frac{1}
{2Q_{\lambda}^{\ast}})\|u\|^{2}
-\mu(\frac{1}{q}-\frac{1}
{2Q_{\lambda}^{\ast}})S^{-\frac{q}{2}}
|\Omega|^{\frac{Q_{\lambda}^{\ast}-q}
{Q_{\lambda}^{\ast}}}\|u\|^{q}.
\end{align}
Let
$$
f(t)=(\frac{1}{2}-\frac{1}
{2Q_{\lambda}^{\ast}})t^{2}-
\mu(\frac{1}{q}-\frac{1}
{2Q_{\lambda}^{\ast}})S^{-\frac{q}{2}}
|\Omega|^{\frac{Q_{\lambda}^{\ast}-q}
{Q_{\lambda}^{\ast}}}t^{q}.
$$
Then by a simple calculation,
we   see that
$$
t_{0}=(\mu(\frac{1}{q}-\frac{1}
{2Q_{\lambda}^{\ast}})S^{-\frac{q}{2}}
|\Omega|^{\frac{Q_{\lambda}^{\ast}-q}
{Q_{\lambda}^{\ast}}})^{\frac{Q_{\lambda}^{\ast}}
{(Q_{\lambda}^{\ast}-1)(2-q)}}
$$
is the minimum value point of
$f(x)$, and the minimum value of
$f(x)$ is
\begin{align*}
f(t_{0})
&=(\mu(\frac{1}{q}
-\frac{1}{2Q_{\lambda}^{\ast}})S^{
-\frac{q}{2}}
|\Omega|^{\frac{Q_{\lambda}^{\ast}-q}
{Q_{\lambda}^{\ast}}})^{\frac{2}{2-q}}
(2q)^{\frac{q}{2-q}}(\frac{q}{2}-1)\\
&<(\mu(\frac{1}{q}
-\frac{1}{2Q_{\lambda}^{\ast}})S^{
-\frac{q}{2}}
|\Omega|^{\frac{Q_{\lambda}^{\ast}-q}
{Q_{\lambda}^{\ast}}})^{\frac{2}{2-q}}
(\frac{Q_{\lambda}^{\ast}}
{Q_{\lambda}^{\ast}-1}q)^{\frac{q}{2-q}}
=D\mu^{\frac{2}{2-q}}.
\end{align*}
Thus,
$$
c>(\frac{1}{2}
-\frac{1}{2Q_{\lambda}^{\ast}})
S_{HG}^{\frac{Q_{\lambda}^{\ast}}
{Q_{\lambda}^{\ast}-1}}-D\mu^{\frac{2}{2-q}},
$$
which contradicts \eqref{3.2}.
Thus,
$J=\emptyset$, and one has
\begin{equation}
\label{3.13}
\int_{\Omega}
\frac{|u_{n}(\eta)|^{Q_{\lambda}^{\ast}}}
{|\eta^{-1}\xi|^{\lambda}}d\eta|u_{n}
(\xi)|^{Q_{\lambda}^{\ast}}\rightarrow
\int_{\Omega}
\frac{|u(\eta)|^{Q_{\lambda}^{\ast}}}
{|\eta^{-1}\xi|^{\lambda}}d
\eta|u(\xi)|^{Q_{\lambda}^{\ast}}\quad
 \mbox{as} \quad n\rightarrow \infty.
\end{equation}
From \eqref{3.2} and \eqref{3.13},
we get
\begin{align}\label{3.14}
&\nonumber\int_{\Omega}
\nabla_H u_{n}\nabla_H \varphi
d\xi-\mu\int_{\Omega}|u_{n}|^{q-1}\varphi
d\xi\\
&-\int_{\Omega}\int_{\Omega}
\frac{|u_{n}(\xi)|^{Q_{\lambda}^{\ast}}
|u_{n}(\eta)|^{Q_{\lambda}^{\ast}-1}
\varphi}{|\eta^{-1}\xi|^{\lambda}}
d\eta d\xi=o(1).
\end{align}
Choose
$\varphi=u$ in \eqref{3.14}. Then
\begin{equation}\label{3.15}
\|u\|^{2}-\mu
\int_{\Omega}|u|^{q}d\xi-\int_{\Omega}
\int_{\Omega}\frac{|u(\xi)|^{Q_{\lambda}^{\ast}}
|u(\eta)|^{Q_{\lambda}^{\ast}}}
{|\eta^{-1}\xi|^{\lambda}}
d\eta d\xi=0.
\end{equation}
By \eqref{3.2} and \eqref{3.13},
we also have
\begin{equation}\label{3.16}
\lim_{n
\to\infty}\|u_{n}\|^{2}-\mu
\int_{\Omega}|u|^{q}d\xi-\int_{\Omega}
\int_{\Omega}\frac{|u(\xi)|^{Q_{\lambda}^{\ast}}
|u(\eta)|^{Q_{\lambda}^{\ast}}}
{|\eta^{-1}\xi|^{\lambda}}
d\eta d\xi=0.
\end{equation}
Combining \eqref{3.15} and \eqref{3.16}, we obtain that $\lim_{n
\to\infty}\|u_{n}\|=\|u\|$. Thus, uniform convexity follows from
$\mathring{S}_1^2(\Omega)$, so we can conclude that
$u_{n}\rightarrow u$ in $\mathring{S}_1^2(\Omega)$. This completes
the proof of Lemma \ref{lem3.1}.
\end{proof}

\begin{lemma}\label{lem3.2}
Let
$1<q<2$. Then there exist
$\Lambda_{0},\rho_{0}>0$ such that if
$\mu\in(0,\Lambda_{0})$, then
$$
\inf_{u\in B\rho_{0}}I_{\mu}(u)<0
$$
 and
$$
I_{\mu}(u)>\frac{1}{2}
g(\rho_{0})\rho_{0}^{q}>0\
 \ \hbox{for all}\ \ u\in S_{\rho_{0}},
 $$
where
$g(s)=\frac{1}{2}s^{2-q}-a_{0}
s^{2Q_{\lambda}^{\ast}-q}$.
\end{lemma}

\begin{proof}[\bf Proof.]
 First, by the Young inequality,
we obtain that
\begin{align} \label{3.17}
I_{\mu}(u) & =\frac{1}{2}\|u\|^{2}
-\frac{\mu}{q}\int_{\Omega}|u|^{q}d\xi
-\frac{1}{2Q_{\lambda}^{\ast}}
\int_{\Omega}\int_{\Omega}
\frac{|u(\xi)|^{Q_{\lambda}^{\ast}}
|u(\eta)|^{Q_{\lambda}^{\ast}}}
{|\eta^{-1}\xi|^{\lambda}}d\eta
d\xi\\ \notag
&\geq\frac{1}{2}\|u\|^{2}
-\frac{\mu}{q}S^{
-\frac{q}{2}}
|\Omega|^{\frac{Q_{\lambda}^{\ast}-q}
{Q_{\lambda}^{\ast}}}\|u\|^{q}
-\frac{C(\lambda,Q)}{2Q_{\lambda}^{\ast}}
\Big ( \int_{\Omega}
|u|^{Q^{\ast}}d\xi
\Big ) ^{\frac{2Q-\lambda}{Q}}\\ \notag
&\geq\frac{1}{2}\|u\|^{2}
-\frac{\mu}{q}S^{
-\frac{q}{2}}
|\Omega|^{\frac{Q_{\lambda}^{\ast}-q}
{Q_{\lambda}^{\ast}}}\|u\|^{q}
-\frac{1}{2Q_{\lambda}^{\ast}
S_{HG}^{Q_{\lambda}^{\ast}}}
\|u\|^{2Q_{\lambda}^{\ast}}\\ \notag
&=\|u\|^{q}
\Big \{\frac{1}{2}\|u\|^{2-q}
-\frac{\mu}{q}S^{
-\frac{q}{2}}
|\Omega|^{\frac{Q_{\lambda}^{\ast}-q}
{Q_{\lambda}^{\ast}}}
-\frac{1}{2Q_{\lambda}^{\ast}
S_{HG}^{Q_{\lambda}^{\ast}}}
\|u\|^{2Q_{\lambda}^{\ast}-q}
\Big \}.
 \end{align}
Let
$a_{0}=\frac{1}{2Q_{\lambda}^{\ast}
S_{HG}^{Q_{\lambda}^{\ast}}}>0$.
Then it follows from \eqref{3.17}
 that
\begin{equation}\label{3.18}
I_{\mu}(u)\geq\|u\|^{q}
\Big \{\frac{1}{2}\|u\|^{2-q}
-\frac{\mu}{q}S^{
-\frac{q}{2}}
|\Omega|^{\frac{Q_{\lambda}^{\ast}-q}
{Q_{\lambda}^{\ast}}}
-a_{0}\|u\|^{2Q_{\lambda}^{\ast}-q}
\Big \}.
\end{equation}
Take
$$
g(s)=\frac{1}{2}s^{2-q}-a_{0}
s^{2Q_{\lambda}^{\ast}-q}.
$$
Then the maximum value point of
$g(s)$ is
$$
\rho_{0}=
\Big ( \frac{2-q}{2a_{0}(2Q_{\lambda}^{\ast}-q)}
\Big ) ^{\frac{1}{2Q_{\lambda}^{\ast}-2}},
$$
and the maximum value of
$g(s)$ is
$$
g(\rho_{0})=\frac{\rho_{0}^{2-q}}{2}
\Big ( 1
-\frac{2-q}{2Q_{\lambda}^{\ast}-q}
\Big ) >0.
$$
Hence, if
$$
\Lambda_{0}=\frac{1}{2}q
S^{\frac{q}{2}}|\Omega|^{\frac{q-Q_{\lambda}^{\ast}}
{Q_{\lambda}^{\ast}}}g(\rho_{0}),
$$
then for all
$\mu\in(0,\Lambda_{0})$,
we have from \eqref{3.18} that
$$
I_{\mu}(u)\geq
\frac{1}{2}g(\rho_{0})\rho_{0}^{q}>0
\ \ \hbox{for all}\ \ u\in S_{\rho_{0}}.
$$
Furthermore, for any
$u\in \mathring{S}_1^2(\Omega)\backslash\{0\}$,
one has
$$
\lim_{s\rightarrow0^{+}}
\frac{I_{\mu}(su)}{s^{q}}=
-\frac{\mu}{q}\int_{\Omega}
|u|^{q}d\xi<0,
$$
which implies that
$u\in B_{\rho_{0}}$ makes
$I_{\mu}(u)<0$. From
this, we can conclude that
$\inf_{u\in B_{\rho_{0}}}I_{\mu}(u)<0,$
and the proof of Lemma \ref{lem3.2}
is complete.
\end{proof}

\begin{lemma}\label{lem3.3}
Let
$1<q<2$ and assume that
$\mu\in(0,\Lambda_{0})$. Then problem
$(1.1)$ has a positive solution
$u_{1}\in \mathring{S}_1^2(\Omega)$
such that
$I_{\mu}(u_{1})<0$.
\end{lemma}

\begin{proof}[\bf Proof.]
Let
$\rho_{0}$ be as in Lemma \ref{lem3.2}
 and set
\begin{equation}\label{3.33}
w=\inf_{u\in B_{\rho_{0}}}
I_{\mu}(u)<0<\inf_{u\in
S_{\rho_{0}}}I_{\mu}(u).
\end{equation}
Note that
$I_{\mu}(|u|)=I_{\mu}(u)$.
According to the Ekeland
variational principle  \cite{a2},
 we know that
$$
I_{\mu}(u_{n})
\leq
\inf_{u\in B_{\rho_{0}}}I_{\mu}(u)+
\frac{1}{n},\quad
I_{\mu}(v)\geq I_{\mu}(u_{n})
-\frac{1}{n}\|v-u_{n}\|
$$
for all
$v\in B_{\rho_{0}}$
and some nonnegative minimizing sequence
$\{u_{n}\}\subset B_{\rho_{0}}$.
From this and \eqref{3.33}, we get
$I_{\mu}'(u)\rightarrow0$ and
$I_{\mu}(u_{n})\rightarrow w$. Because
of
$u_{n}\geq0$ and
$\|u_{n}\|\leq\rho_{0}$, there is
$u_{1}\in
B_{\rho_{0}}$ and
$u_{1}\geq0$ satisfying
$u_{n}\rightharpoonup
u_{1}$ in
$\mathring{S}_1^2(\Omega)$ as
$n\rightarrow\infty$. It
follows from Lemma \ref{lem3.1} that
$u_{n}\rightarrow u_{1}$ in
$\mathring{S}_1^2(\Omega)$ and
$$
w=\lim_{n\rightarrow\infty}
I_{\mu}(u_{n})=
I_{\mu}(u_{1})<0.
$$
 Hence,
we have
$u_{1}\geq0$ and
$u_{1}\not\equiv0$. Moreover,
$u_{1}$ is a
solution of problem \eqref{1.1}, that is
$$
{-\Delta_H u_{1} }=\mu|u_{1}|^{q-1}+
\int_{\Omega}
\frac{|u_{1}(\eta)|^{Q_{\lambda}^{\ast}}}
{|\eta^{-1}\xi|^{\lambda}}
d\eta|u_{1}|^{Q_{\lambda}^{\ast}-1}.
$$
The maximum principle (see Bony \cite{bo}) implies that $u_{1}>0$ in
$\Omega$. Thus, $u_{1}$ is a positive solution of problem
\eqref{1.1}. This completes the proof of Lemma \ref{lem3.3}.
\end{proof}

\subsection{The existence of the second
 positive solution of problem \eqref{1.1}}

\begin{lemma}\label{lem3.4}
Let
$1<q<2$ and
$\mu\in(0,\Lambda_{0})$. Then
$I_{\mu}(u)>0$ for all
$u\in S_{\rho_{0}}$. Moreover, there is
$e\in
\mathring{S}_1^2(\Omega)\backslash B_{\rho_{0}}$
satisfying
$I_{\mu}(e)<0$, where
$\Lambda_{0},\rho_{0}$ are as in Lemma
$\ref{lem3.2}.$
\end{lemma}

\begin{proof}[\bf Proof.]
It is evident that Lemma \ref{lem3.2}
 proves the first assertion.
Therefore, we only need to
prove the rest of Lemma \ref{lem3.4}. Let
$u\in \mathring{S}_1^2(\Omega)\backslash\{0\}$.
Then one has
\begin{align}\label{3.34}
I_{\mu}(su)&=\nonumber\frac{s^{2}}{2}\|u\|^{2}
-\frac{\mu
s^{q}}{q}\int_{\Omega}
|u|^{q}d\xi
-\frac{s^{2Q_{\lambda}^{\ast}}}
{2Q_{\lambda}^{\ast}}
\int_{\Omega}\int_{\Omega}
\frac{|u(\xi)|^{Q_{\lambda}^{\ast}}
|u(\eta)|^{Q_{\lambda}^{\ast}}}
{|\eta^{-1}\xi|^{\lambda}}d\eta d\xi\\
&\leq\frac{s^{2}}{2}\|u\|^{2}
-\frac{C(\lambda,Q)s^{2Q_{\lambda}^{\ast}}}
{2Q_{\lambda}^{\ast}}
\Big ( \int_{\Omega}
|u|^{Q^{\ast}}d\xi
 \Big ) ^{\frac{2Q-\lambda}{Q}}
\rightarrow -\infty, \ \hbox{as} \
s\rightarrow +\infty.
\end{align}
Thus there exists
$e\in \mathring{S}_1^2(\Omega)\backslash
B_{\rho_{0}}$ satisfying
$I_{\mu}(e)<0$. This completes the proof of
Lemma \ref{lem3.4}.
\end{proof}

\begin{lemma}\label{lem3.5}
Let
$1<q<2$ and assume that
$\upsilon_{\gamma}$ is defined by \eqref{2.8}.
 Then there exists
$\Lambda_{1}>0$ such that
$\mu\in(0,\Lambda_{1})$, and
\begin{equation}\label{3.19}
\sup_{s\geq0}I_{\mu}(u_{1}+s\upsilon_{\gamma})<
\Big ( \frac{1}{2}
-\frac{1}{2Q_{\lambda}^{\ast}}
\Big ) S_{HG}^{\frac{Q_{\lambda}^{\ast}}
{Q_{\lambda}^{\ast}-1}}-D\mu^{\frac{2}{2-q}},
\end{equation}
where
$u_{1}$ is the positive solution from
Lemma $\ref{lem3.4}$ and D is from \eqref{3.1}.
\end{lemma}

\begin{proof}[\bf Proof.]
Because
$u_{1}$ is the positive solution from
Lemma \ref{lem3.4},
there exist positive constants
$t$ and
$T$
satisfying
$t\leq u_{1}(\xi)\leq T$ for any
$\xi\in \sup\nu$, where
$\nu$ is as in \eqref{2.8}. Moreover, one has
$I_{\mu}'(u_{1})u_{1}=0$ and
$I_{\mu}(u_{1})<0$.

Next, it is easy to prove that for any
$a,b>0$, we have
\begin{equation}\label{3.20}
(a+b)^{\sigma}\geq a^{\sigma}+\sigma
a^{\sigma-1}b,\quad 1<\sigma<2
\end{equation}
and
\begin{equation}\label{3.20-1}
(a_{1}+b_{1})^{Q_{\lambda}^{\ast}}
(a_{2}+b_{2})^{Q_{\lambda}^{\ast}}\geq
a_{1}^{\sigma}a_{2}^{\sigma}
+b_{1}^{\sigma}b_{2}^{\sigma} +2\sigma
a_{1}^{\sigma}a_{2}^{\sigma-1}b_{2}
+2\sigma
a_{1}^{\sigma}a_{2}b_{2}^{\sigma-1},\
  2\leq \sigma.
\end{equation}
Hence, for any
$s\geq0$, we have
\begin{align}\label{3.21}
&\nonumber
I_{\mu}(u_{1}+s\upsilon_{\gamma})= I_{\mu}
(u_{1})+\frac{s^{2}}{2}
\|\upsilon_{\gamma}|^{2}+sI'_{\mu}(u_{1})
\upsilon_{\gamma}\\
&\nonumber
-\frac{\mu}{q}\int_{\Omega}
\Big [(u_{1}+s\upsilon_{\gamma})^{q}
-u_{1}^{q}-qsu_{1}^{q-1}
\upsilon_{\gamma}
\Big ]d\xi\\
&\nonumber
-\frac{1}{2Q_{\lambda}^{\ast}}
\int_{\Omega}\int_{\Omega}
\Big [\frac{|u_{1}(\xi)+
s\upsilon_{\gamma}(\xi)|^{Q_{\lambda}^{\ast}}
|u_{1}(\eta)
+s\upsilon_{\gamma}(\eta)|^{Q_{\lambda}^{\ast}}}
{|\eta^{-1}\xi|^{\lambda}}
 \\
&\nonumber
-\frac{|u_{1}(\xi)|^{Q_{\lambda}^{\ast}}
|u_{1}(\eta)|^{Q_{\lambda}^{\ast}}}
{|\eta^{-1}\xi|^{\lambda}}-
2Q_{\lambda}^{\ast}s
\frac{|u_{1}(\xi)|^{Q_{\lambda}^{\ast}}
|u_{1}(\eta)|^{Q_{\lambda}^{\ast}-1}
|\upsilon_{\gamma}(\eta)|}
{|\eta^{-1}\xi|^{\lambda}}
\Big ]d\eta d\xi\\
&\leq\nonumber\frac{s^{2}}{2}
\|\upsilon_{\gamma}\|^{2}
-\frac{s^{2Q_{\lambda}^{\ast}}}
{2Q_{\lambda}^{\ast}}\int_{\Omega}
\int_{\Omega}
\frac{|\upsilon_{\gamma}(\xi)|^{Q_{\lambda}^{\ast}}
|\upsilon_{\gamma}(\eta)|^{Q_{\lambda}^{\ast}}}
{|\eta^{-1}\xi|^{\lambda}}d\eta
d\xi\\
&\nonumber
-s^{Q_{\lambda}^{\ast}-1}
\int_{\Omega}\int_{\Omega}
\frac{|u_{1}(\xi)|^{Q_{\lambda}^{\ast}}
|u_{1}(\eta)
||\upsilon_{\gamma}(\eta)|^{Q_{\lambda}^{\ast}-1}}
{|\eta^{-1}\xi|^{\lambda}}d\eta
d\xi\\
&\nonumber\leq\frac{s^{2}}{2}
\|\upsilon_{\gamma}\|^{2}
-\frac{C(\lambda,Q)s^{2Q_{\lambda}^{\ast}}}
{2Q_{\lambda}^{\ast}}
(\int_{\Omega}|\upsilon_{\gamma}
|^{Q^{\ast}}d\xi)^{\frac{2Q-\lambda}{Q}}\\
&
-ts^{Q_{\lambda}^{\ast}-1}
\int_{\Omega}\int_{\Omega}
\frac{|u_{1}(\xi)|^{Q_{\lambda}^{\ast}}
|\upsilon_{\gamma}(\eta)|^{Q_{\lambda}^{\ast}-1}}
{|\eta^{-1}\xi|^{\lambda}}d\eta
d\xi.
\end{align}
Let
\begin{align*}
\phi(s) & = \frac{s^{2}}{2}
\|\upsilon_{\gamma}\|^{2}
-\frac{C(\lambda,Q)s^{2Q_{\lambda}^{\ast}}}
{2Q_{\lambda}^{\ast}}
 \Big ( \int_{\Omega}
 |\upsilon_{\gamma}|^{Q^{\ast}}d\xi
 \Big ) ^{\frac{2Q-\lambda}{Q}}\\
&
-ts^{Q_{\lambda}^{\ast}-1}
\int_{\Omega}\int_{\Omega}
\frac{|u_{1}(\xi)|^{Q_{\lambda}^{\ast}}
|\upsilon_{\gamma}
(\eta)|^{Q_{\lambda}^{\ast}-1}}
{|\eta^{-1}\xi|^{\lambda}}d\eta
d\xi.
\end{align*}
The definition of
$\phi(s)$ enables us to obtain
$\phi(0)=0$ and
$\phi(s)\rightarrow-\infty,$ as
$s\rightarrow+\infty$. Thus, we can
find
$s_{\gamma}>0$ and positive constants
$s_{1}, s_{2}$
independent of
$\gamma, \mu$, satisfying
\begin{equation}\label{3.22}
\phi(s_{\gamma})=\sup_{s\geq0}\phi(s),
\quad \phi'(s_{\gamma})=0
\end{equation}
and
\begin{equation}\label{3.23}
0<s_{1}\leq s_{\gamma}\leq s_{2}<\infty.
\end{equation}
Therefore, one has
\begin{align}\label{3.24}
&s_{\gamma}\|\upsilon_{\gamma}\|^{2}
-C(\lambda,Q)
s_{\gamma}^{2Q_{\lambda}^{\ast}-1}
 \Big ( \int_{\Omega}
 |\upsilon_{\gamma}|^{Q^{\ast}}d\xi
 \Big ) ^{\frac{2Q-\lambda}{Q}}\\
&\ -t(Q_{\lambda}^{\ast}-1)
s_{\gamma}^{Q_{\lambda}^{\ast}-2}
\int_{\Omega}\int_{\Omega}
\frac{|u_{1}(\xi)|^{Q_{\lambda}^{\ast}}
|\upsilon_{\gamma}
(\eta)|^{Q_{\lambda}^{\ast}-1}}
{|\eta^{-1}\xi|^{\lambda}}d\eta
d\xi=0
\end{align}
and
\begin{align}\label{3.25}
&
\|\upsilon_{\gamma}\|^{2}-
C(\lambda,Q)(2Q_{\lambda}^{\ast}-1)
s_{\gamma}^{2Q_{\lambda}^{\ast}-2}
\Big ( \int_{\Omega}
|\upsilon_{\gamma}|^{Q^{\ast}}d\xi
\Big ) ^{\frac{2Q-\lambda}{Q}}\\
&
-t(Q_{\lambda}^{\ast}-1)
(Q_{\lambda}^{\ast}-2)
s_{\gamma}^{Q_{\lambda}^{\ast}-3}
\int_{\Omega}\int_{\Omega}
\frac{|u_{1}(\xi)|^{Q_{\lambda}^{\ast}}
|\upsilon_{\gamma}
(\eta)|^{Q_{\lambda}^{\ast}-1}}
{|\eta^{-1}\xi|^{\lambda}}d\eta d\xi<0.
\end{align}
From \eqref{3.25},
we obtain that there exists
$s_{1}>0$ (independent of
$\gamma, \mu$) satisfying
$0<s_{1}\leq s_{\gamma}$.

Next, from \eqref{3.24} one has
\begin{align}\label{3.26}
\begin{split}
&
\frac{\|\upsilon_{\gamma}\|^{2}}
{s_{\gamma}^{2Q_{\lambda}^{\ast}-2}}
-C(\lambda,Q)
 \Big ( \int_{\Omega}
 |\upsilon_{\gamma}|^{Q^{\ast}}d\xi
 \Big ) ^{\frac{2Q-\lambda}{Q}}\\
&
-\frac{t(Q_{\lambda}^{\ast}-1)}
{s_{\gamma}^{Q_{\lambda}^{\ast}+1}}
\int_{\Omega}\int_{\Omega}
\frac{|u_{1}(\xi)|^{Q_{\lambda}^{\ast}}
|\upsilon_{\gamma}
(\eta)|^{Q_{\lambda}^{\ast}-1}}
{|\eta^{-1}\xi|^{\lambda}}d\eta
d\xi=0.
\end{split}
\end{align}
This implies that
$s_{\gamma}$ has an upper
bound for
$\gamma>0$ small enough.
If not, making
$s_{\gamma}\rightarrow\infty$
 in \eqref{3.26},
one gets $\int_{\Omega}|\upsilon_{\gamma}|^{Q^{\ast}}d\xi=0$, which
contradicts Lemma \ref{lem2.2} for $\gamma$ small enough. It follows
from \eqref{3.21}, \eqref{3.22}, \eqref{3.23}, and Lemmas
\ref{lem2.1}, and \ref{lem2.2} that
\begin{align}
\label{3.27}
&\sup_{s\geq0} I_{\mu}(u_{1}+
s\upsilon_{\gamma})
\leq
\nonumber\sup_{s\geq0}\Phi(s)
=\Phi(s_{\gamma})\\
&\leq\nonumber\sup_{s\geq0}
 \Big \{\frac{s^{2}}{2}
\|\upsilon_{\gamma}\|^{2}
-\frac{C(\lambda,Q)
s^{2Q_{\lambda}^{\ast}}}{2Q_{\lambda}^{\ast}}
 \Big ( \int_{\Omega}
 |\upsilon_{\gamma}|^{Q^{\ast}}d\xi
 \Big ) ^{\frac{2Q-\lambda}{Q}}
 \Big\}\\
&\ \ \nonumber -ts^{Q_{\lambda}^{\ast}-1}
\int_{\Omega}\int_{\Omega}
\frac{|u_{1}(\xi)|^{Q_{\lambda}^{\ast}}
|\upsilon_{\gamma}
(\eta)|^{Q_{\lambda}^{\ast}-1}}
{|\eta^{-1}\xi|^{\lambda}}d\eta d\xi\\
&\leq\nonumber\sup_{s\geq0}
 \Big \{\frac{s^{2}}{2}
 \|\upsilon_{\gamma}\|^{2}
-\frac{C(\lambda,Q)
s^{2Q_{\lambda}^{\ast}}}{2Q_{\lambda}^{\ast}}
 \Big ( \int_{\Omega}
 |\upsilon_{\gamma}|^{Q^{\ast}}d\xi
 \Big ) ^{\frac{2Q-\lambda}{Q}}
 \Big\}\\
&\ \
\nonumber-t
s^{Q_{\lambda}^{\ast}-1}C(\lambda,Q)
\Big ( \int_{\Omega}
|u_{1}(\xi)|^{Q^{*}}d\xi
\Big ) ^{\frac{2Q-\lambda}{2Q}}
 \Big ( \int_{\Omega}
 |\upsilon_{\gamma}(\eta)|^{2}d\eta
 \Big ) ^{\frac{Q-\lambda+2}{2(Q-2)}}
 \\
&\leq\nonumber\sup_{s\geq0}
\Big \{\frac{s^{2}}{2}
S_{HG}^{\frac{2Q-\lambda}
{Q-\lambda+2}}
-\frac{C(\lambda,Q)s^{2Q_{\lambda}^{\ast}}}
{2Q_{\lambda}^{\ast}}
 \Big ( S^{\frac{Q}{Q-\lambda+2}}
 C(\lambda,Q)^{\frac{-Q}{Q-\lambda+2}}
 \Big ) ^{\frac{2Q-\lambda}{Q}}
 \Big \}\\
&\ \ \nonumber-C_{1}
(\gamma^{-2})^{\frac{Q-\lambda+2}{2(Q-2)}}\\
&
=\nonumber\sup_{s\geq0}
\Big \{\frac{s^{2}}{2}
S_{HG}^{\frac{Q_{\lambda}^{\ast}}
{Q_{\lambda}^{\ast}-1}}
-\frac{C(\lambda,Q)
s^{2Q_{\lambda}^{\ast}}}
{2Q_{\lambda}^{\ast}}
S_{HG}^{\frac{Q_{\lambda}^{\ast}}
{Q_{\lambda}^{\ast}-1}}
\Big \}-C_{1}\gamma^{
-\frac{Q-\lambda+2}{Q-2}}\\
&<
\Big  ( \frac{1}{2}
-\frac{1}{2Q_{\lambda}^{\ast}}
\Big  ) S_{HG}^{\frac{Q_{\lambda}^{\ast}}
{Q_{\lambda}^{\ast}-1}}-C_{1}\gamma^{
-\frac{Q-\lambda+2}{Q-2}},
\end{align}
where
$C_{1}>0$ is independent of
$\gamma$ and
$\mu$.

Let
$\gamma^{
-\frac{Q-\lambda+2}{Q-2}}=\mu^{\frac{q}{2-q}}$
and
$\Lambda_{1}=\frac{C_{1}}{D}$.
Then for any
$\mu\in(0,\Lambda_{1})$, one has
\begin{align}\label{3.28}
\begin{split}
C_{1}\gamma^{
-\frac{Q-\lambda+2}{Q-2}}>
D\mu^{\frac{2}{2-q}}.
\end{split}
\end{align}
By \eqref{3.27} and \eqref{3.28},  equation \eqref{3.19} holds if
$\mu\in(0,\Lambda_{1})$. This completes the proof of Lemma
\ref{lem3.5}.
\end{proof}
From the above discussion,
we get the following result.

\begin{lemma}\label{lem3.6}
Let
$1<q<2$. Then there exists
$\mu_{\ast}>0$ such that for all
$\mu\in(0,\Lambda_{\ast})$,
problem \eqref{1.1} has a positive
solution
$u_{2}\in \mathring{S}_1^2(\Omega)$ satisfying
$I_{\mu}(u_{2})>0$.
\end{lemma}

\subsection{Existence of a  positive
 ground state solution of problem
\eqref{1.1}}

 In this subsection,
 we will show that problem
\eqref{1.1} has a positive ground
state solution. Indeed, let
$$
\mathcal {N}=\{u\in \mathring{S}_1^2(\Omega):
u\neq0,\langle I'_{\mu}(u),u\rangle=0\}.
$$
Since any nontrivial solution of
 problem \eqref{1.1} belongs to
$\mathcal {N}$, we can set
$\tau=\inf_{u\in \mathcal
{N}}I_{\mu}(u).$ Clearly, if
$u\in \mathcal {N}$, one also has
$|u|\in \mathcal {N}$ and
$I_{\mu}(|u|)=I_{\mu}(u)$,
and therefore
we can consider a nonnegative
 minimizing sequence
$\{u_{n}\}\subset
\mathcal {N}$ and such that
\begin{equation}\label{3.29}
I_{\mu}(u_{n})\rightarrow\tau
\quad as\  n\rightarrow\infty.
\end{equation}
By Lemma \ref{lem3.3},
$\tau<0$ and
$\{u_{n}\}$ is
bounded in
$\mathring{S}_1^2(\Omega)$. More generally,
suppose that
$u_{n}\rightharpoonup u_{\lambda}$ weakly in
$\mathring{S}_1^2(\Omega)$ and
$u_{n}\rightarrow u_{\lambda}$
strongly in
$L^{p}(\Omega)$ with
$1<p<Q^{*}$. Thus,
$u_{\lambda}\neq0$. In fact, if
$u_{\lambda}=0$ and
$l=\lim_{n\rightarrow\infty}\|u_{n}\|$,
 then since
$u_{n}\in
\mathcal {M}$, one has
\begin{align}\label{3.30}
\|u_{n}\|^{2}
&=I_{\mu}'(u_{n})[u_{n}]+\mu\int_{\Omega}
|u_{n}|^{q}d\xi+\int_{\Omega}\int_{\Omega}
\frac{|u_{n}(\xi)|^{Q_{\lambda}^{\ast}}
|u_{n}(\eta)|^{Q_{\lambda}^{\ast}}}
{|\eta^{-1}\xi|^{\lambda}}
d\eta d\xi \notag \\
&\leq
C(\lambda,Q)
\Big  ( \int_{\Omega}
|u_{n}|^{Q^{\ast}}d\xi
\Big  ) ^{\frac{2Q-\lambda}{Q}}+o(1)\leq
S_{HG}^{-Q_{\lambda}^{\ast}}
\|u_{n}\|^{2Q_{\lambda}^{\ast}}+o(1).
\end{align}
From this, one can infer that either $l=0$ or $l\geq
S_{HG}^{\frac{Q_{\lambda}^{\ast}} {2Q_{\lambda}^{\ast}-2}}$.
 Besides, from \eqref{3.29}, one has
\begin{align}\label{3.31}
\tau
&
=\nonumber\lim_{n\rightarrow\infty}
\Big \{\frac{1}{2}\|u_{n}\|^{2}
-\frac{\mu}{q}\int_{\Omega}
|u_{n}|^{q}d\xi
-\frac{1}{2Q_{\lambda}^{\ast}}
\int_{\Omega}\int_{\Omega}
\frac{|u_{n}(\xi)|^{Q_{\lambda}^{\ast}}
|u_{n}(\eta)|^{Q_{\lambda}^{\ast}}}
{|\eta^{-1}\xi|^{\lambda}}
d\eta d\xi
\Big \}\\
&=
\Big  ( \frac{1}{2}
-\frac{1}{2Q_{\lambda}^{\ast}}
\Big  ) \lim_{n\rightarrow\infty}
\|u_{n}\|^{2}=
\Big  ( \frac{1}{2}
-\frac{1}{2Q_{\lambda}^{\ast}}
\Big  ) l^{2}.
\end{align}
If
$l=0$, then from \eqref{3.31},
 we get
$\tau=0$, which is a
contradiction. Thus
$$
\tau\geq
\Big  ( \frac{1}{2}
-\frac{1}{2Q_{\lambda}^{\ast}}
\Big  ) S_{HG}^{\frac{Q_{\lambda}^{\ast}}
{Q_{\lambda}^{\ast}-1}}.
$$
It follows from Lemma \ref{lem3.5} that
$$
\Big  ( \frac{1}{2}
-\frac{1}{2Q_{\lambda}^{\ast}}
\Big  ) S_{HG}^{\frac{Q_{\lambda}^{\ast}}
{Q_{\lambda}^{\ast}-1}}\leq\tau
<
\Big  ( \frac{1}{2}
-\frac{1}{2Q_{\lambda}^{\ast}}
\Big  )
S_{HG}^{\frac{Q_{\lambda}^{\ast}}
{Q_{\lambda}^{\ast}-1}}
-D\mu^{\frac{2}{2-q}},
$$
which is also a contradiction.
Therefore, we must have
$u_{\lambda}\neq0$ in
$\mathring{S}_1^2(\Omega)$.

On the other hand,
$u_{n}\rightarrow u_{\lambda}$ in
$\mathring{S}_1^2(\Omega)$ is
derived from Lemma \ref{lem3.1}.
 In other words,
$u_{\lambda}$ is a positive
solution of problem (1.1) and
$I_{\mu}(u_{\lambda})\geq\tau$.

Next, we show that
$I_{\mu}(u_{\lambda})\leq\tau$.
Indeed, by the
Fatou Lemma and \eqref{3.28},
 we get
\begin{align}\label{3.32}
\tau
&=\nonumber\lim_{n\rightarrow\infty}
\Big \{I_{\mu}(u_{n})
-\frac{1}{2Q_{\lambda}^{\ast}}I'_{\mu}(u_{n})u_{n}
\Big \}\\
&
=\nonumber\lim_{n\rightarrow\infty}
\Big \{
\Big  ( \frac{1}{2}
-\frac{1}{2Q_{\lambda}^{\ast}}
\Big  ) \|u_{n}\|^{2}-\mu
\Big  ( \frac{1}{q}
-\frac{1}{2Q_{\lambda}^{\ast}}
\Big  ) \int_{\Omega}u_{n}^{q}d\xi
\Big \}\\
&\geq
\Big  ( \frac{1}{2}
-\frac{1}{2Q_{\lambda}^{\ast}}
\Big  ) \|u_{\lambda}\|^{2}-\mu
\Big  ( \frac{1}{q}
-\frac{1}{2Q_{\lambda}^{\ast}}
\Big  ) \int_{\Omega}u_{\lambda}^{q}d\xi.
\end{align}
Furthermore, because
$u_{\lambda}$ is a positive solution of problem
\eqref{1.1}, we have
\begin{align*}
\begin{split}
I_{\mu}(u_{\lambda})&=I_{\mu}(u_{\lambda})
-\frac{1}{2Q_{\lambda}^{\ast}}
I'_{\mu}(u_{\lambda})u_{\lambda}\\
&=
\Big  ( \frac{1}{2}
-\frac{1}{2Q_{\lambda}^{\ast}}
\Big  ) \|u_{\lambda}\|^{2}-\mu
\Big  ( \frac{1}{q}
-\frac{1}{2Q_{\lambda}^{\ast}}
\Big  ) \int_{\Omega}u_{\lambda}^{q}d\xi.
\end{split}
\end{align*}
From \eqref{3.32}, we have
$I_{\lambda}(u_{\lambda})\leq\tau$  and
$I_{\lambda}(u_{\lambda})=\tau$ and
$u_{\lambda}\neq0$. This means
that
$u_{\lambda}$ is a positive
ground state solution of problem
\eqref{1.1}. Consequently,
invoking Lemmas \ref{lem3.3} and
\ref{lem3.6} completes the
proof of Theorem \ref{the1.1}. \qed

\section{High perturbations of problem \eqref{1.1}}

This section focuses on the
proof of Theorem \ref{the1.2}.
To this end, we shall apply
the general Mountain Pass Theorem.

\begin{theorem}\label{the4.1}
{\rm (see Rabinowitz \cite{Rabinowitz})} Let E be a real Banach
space and $I\in C^{1}(E,R)$ satisfying $(PS)$ condition. Suppose
that $I(0) = 0$ and that

$(i)$  there are constants
$\rho,\alpha>0$ satisfying
$I(u)|_{\partial B_{\rho}}\geq \alpha$;

$(ii)$  there exists
$e\in E\backslash\overline{B_{\rho}}$ satisfying
$I(e)\leq0$.

Then I has a critical value
$c\geq\alpha$. Moreover,
$$
c=\inf_{h\in\Gamma}\max_{0\leq
t\leq1}I(h(t))\geq\alpha,
$$
where
$$
\Gamma=\{h\in C([0,1],E):h(0)=1,h(1)=e\}.
$$
\end{theorem}

Next, we prove that the geometric properties
$(i)$ and $(ii)$ of Theorem \ref{the4.1}
 are satisfied by
$I_{\mu}$.

\begin{lemma}\label{lem4.1}
Let
$2<q<Q_\lambda^\ast$.
Then the properties $(i)$ and $(ii)$ of
Theorem $\ref{the4.1}$ are satisfied by the
energy functional
$I_{\mu}$.
\end{lemma}

\begin{proof}[\bf Proof.]
From \eqref{3.17}, we have
\begin{align}\label{4.1}
\begin{split}
I_{\mu}(u)\geq\frac{1}{2}\|u\|^{2}
-\frac{\mu}{q}S^{
-\frac{q}{2}}|\Omega|^{\frac{Q_{\lambda}^{\ast}-q}
{Q_{\lambda}^{\ast}}}\|u\|^{q}
-\frac{1}{2Q_{\lambda}^{\ast}S_{HG}^{Q_{\lambda}^{\ast}}}
\|u\|^{2Q_{\lambda}^{\ast}}.\\
\end{split}
\end{align}
Now, we can take
$\rho,\alpha>0$ satisfying
$I_{\mu}(u)\geq \alpha$ for all
$u\in \partial{B_{\rho}}$, since
$2<q$,
$2<2Q_{\lambda}^{\ast}$.
Thus, property (i) of
Theorem \ref{the4.1} holds.

Now, we show that property (ii) of
Theorem \ref{the4.1} also holds.
From \eqref{3.34}, one has
\begin{align}\label{4.2}
\begin{split}
I_{\mu}(su)&=\frac{s^{2}}{2}\|u\|^{2}
-\frac{\mu s^{q}}{q}\int_{\Omega}
|u|^{q}d\xi
-\frac{s^{2Q_{\lambda}^{\ast}}}
{2Q_{\lambda}^{\ast}}
\int_{\Omega}\int_{\Omega}
\frac{|u(\xi)|^{Q_{\lambda}^{\ast}}
|u(\eta)|^{Q_{\lambda}^{\ast}}}
{|\eta^{-1}\xi|^{\lambda}}d\eta d\xi\\
&\leq\frac{s^{2}}{2}\|u\|^{2}
-\frac{C(\lambda,Q)s^{2Q_{\lambda}^{\ast}}}
{2Q_{\lambda}^{\ast}}
\Big  ( \int_{\Omega}
|u|^{Q^{\ast}}d\xi
\Big  ) ^{\frac{2Q-\lambda}{Q}}.
\end{split}
\end{align}
Since
$2<2Q_{\lambda}^{\ast}$, we can deduce that
$I_{\mu}(s_{0}u)<0$ and
$s_{0}\|u\|>\rho$ for
$s_{0}$ large enough.
Let
$e=s_{0}u$.  Hence,
$e$ is the desired function and the proof of
property (ii) of Theorem \ref{the4.1} is complete.
\end{proof}

\begin{lemma}\label{lem3.7}
Let
$2<q<Q_\lambda^\ast$. Then for all
\begin{equation}\label{3.35}
c_{\mu}<
\Big  ( \frac{1}{2}
-\frac{1}{q}
\Big  ) S_{HG}^{\frac{Q_{\lambda}^{\ast}}
{Q_{\lambda}^{\ast}-1}},
\end{equation}
$I_{\mu}$ satisfies the
$(PS)_{c_{\mu}}$ condition.
\end{lemma}

\begin{proof}[\bf Proof.]
We assume that
$\{u_{n}\}\subset \mathring{S}_1^2(\Omega)$
satisfies
\begin{equation}\label{3.36}
I_{\mu}(u_{n})\rightarrow c_{\mu},
\ \ I_{\mu}'(u_{n})\rightarrow 0
\quad \mbox{as} \quad n\rightarrow\infty,
\end{equation}
where
$c_{\mu}$ is as in \eqref{3.36}. It follows that
\begin{align}\label{3.37}
1+c_{\mu}+o(\|u_{n}\|)\nonumber\geq&
I_{\mu}(u_{n})
-\frac{1}{q}I_{\mu}'(u_{n})u_{n}=
\Big  ( \frac{1}{2}
-\frac{1}{q}
\Big  ) \|u_{n}\|^{2}\\
&\nonumber
+
\Big  ( \frac{1}{q}
-\frac{1}{2Q_{\lambda}^{\ast}}
\Big  ) \int_{\Omega}\int_{\Omega}
\frac{|u_{n}(\xi)|^{Q_{\lambda}^{\ast}}
|u_{n}(\eta)|^{Q_{\lambda}^{\ast}}}
{|\eta^{-1}\xi|^{\lambda}}
d\eta d\xi\\
\geq&
\Big  ( \frac{1}{2}
-\frac{1}{q}
\Big  ) \|u_{n}\|^{2}.
\end{align}
From this we can derive that $\{u_{n}\}$ is bounded in
$\mathring{S}_1^2(\Omega)$.

The rest of the proof is similar to the proof of Lemma \ref{lem3.1}.
We have $\mu_{j}\geq S_{HG}^{\frac{Q_{\lambda}^{\ast}}
{Q_{\lambda}^{\ast}-1}}$ and by using the concentration compactness
principle, we have
\begin{align*}
\label{3.38}
c_{\mu}& =\lim_{n \to\infty}
\Big \{I_{\mu}(u_{n})
-\frac{1}{q}I_{\mu}'(u_{n})u_{n}
\Big \}\\
&\nonumber=\lim_{n
\to\infty}
\Big \{(\frac{1}{2}
-\frac{1}{q})\|u_{n}\|^{2}+
\Big  ( \frac{1}{q}
-\frac{1}{2Q_{\lambda}^{\ast}}
\Big  ) \int_{\Omega}\int_{\Omega}
\frac{|u_{n}(\xi)|^{Q_{\lambda}^{\ast}}
|u_{n}(\eta)|^{Q_{\lambda}^{\ast}}}
{|\eta^{-1}\xi|^{\lambda}}
d\eta d\xi
\Big \}\\
&\nonumber\geq\lim_{n
\to\infty}
\Big \{(\frac{1}{2}
-\frac{1}{q})\|u_{n}\|^{2}
\Big \}\geq(\frac{1}{2}
-\frac{1}{q})\mu_{j}+(\frac{1}{2}
-\frac{1}{q})\|u\|^{2}>
\Big  ( \frac{1}{2}
-\frac{1}{q}
\Big  ) S_{HG}^{\frac{Q_{\lambda}^{\ast}}
{Q_{\lambda}^{\ast}-1}}.
\end{align*}
From \eqref{3.35}, we have $J=\emptyset$. As in the discussion of
Lemma \ref{lem3.1}, one knows that $u_{n}\rightarrow u$ in
$\mathring{S}_1^2(\Omega)$. This completes the proof of Lemma
\ref{lem3.7}.
\end{proof}

\begin{proof}[\bf Proof of Theorem \ref{the1.2}]
We claim that
\begin{equation}\label{4.3}
0<c_{\mu}=\inf_{h\in\Gamma}\max_{0\leq
t\leq1}I_{\mu}(h(t))<
\Big  ( \frac{1}{2}
-\frac{1}{q}
\Big  ) S_{HG}^{\frac{Q_{\lambda}^{\ast}}
{Q_{\lambda}^{\ast}-1}}.
\end{equation}
Assume that \eqref{4.3} holds. Then Lemmas \ref{lem3.7} and
\ref{4.1} and Theorem \ref{the4.1} assure the existence of
nontrivial critical points of $I_{\mu}$.

In order to prove \eqref{4.3}, we choose
$v_{0}\in \mathring{S}_1^2(\Omega)$
such that
$$\|v_{0}\|=1\quad \mbox{and}\quad
\lim_{s\rightarrow+\infty}I_{\mu}(sv_{0})=-\infty.$$
Then
$$
\sup_{s\geq0}I_{\mu}(sv_{0})=
I_{\mu}(s_{\mu}v_{0})
\ \text{ for some
$s_{\mu}>0$.}
$$
  So,
$s_{\mu}$ satisfies
\begin{equation}\label{4.4}
\|s_{\mu}v_{0}\|^{2}=\mu\int_{\Omega}
|s_{\mu}v_{0}|^{q}d\xi
+\int_{\Omega}\int_{\Omega}
\frac{|s_{\mu}v_{0}(\xi)|^{Q_{\lambda}^{\ast}}
|s_{\mu}v_{0}(\eta)|^{Q_{\lambda}^{\ast}}}
{|\eta^{-1}\xi|^{\lambda}}
d\eta d\xi.
\end{equation}
Now, we demonstrate that
$\{s_{\mu}\}_{\mu>0}$ is bounded.
Indeed, suppose that the following hypothesis
$s_{\mu}\geq1$ is satisfied for all
$\mu>0$.
Furthermore, we can deduce from
 \eqref{4.4} that
\begin{align}\label{4.5}
s_{\mu}^{q}&\geq\|s_{\mu}v_{0}\|^{2}=
\mu\int_{\Omega}|s_{\mu}v_{0}|^{q}d\xi
+\int_{\Omega}\int_{\Omega}
\frac{|s_{\mu}v_{0}(\xi)|^{Q_{\lambda}^{\ast}}
|s_{\mu}v_{0}(\eta)|^{Q_{\lambda}^{\ast}}}
{|\eta^{-1}\xi|^{\lambda}}d\eta d\xi\\
&\geq
s_{\mu}^{2Q_{\lambda}^{\ast}}
\int_{\Omega}\int_{\Omega}
\frac{|v_{0}(\xi)|^{Q_{\lambda}^{\ast}}
|v_{0}(\eta)|^{Q_{\lambda}^{\ast}}}
{|\eta^{-1}\xi|^{\lambda}}d\eta
d\xi.
\end{align}
It follows from \eqref{4.5} that
$\{s_{\mu}\}_{\mu>0}$ is bounded, since
$2<2Q_{\lambda}^{\ast}$.

Next, we demonstrate that $s_{\mu}\rightarrow0$ as
$\mu\rightarrow\infty$. Suppose to the contrary, there are
$s_{\mu}>0$ and a sequence $(\mu_{n})_{n}$ with
$\mu_{n}\rightarrow\infty$ as $n\rightarrow\infty$ that satisfying
$s_{\mu_{n}}\rightarrow s_{\mu}$ as $n\rightarrow\infty$. Invoking
the Lebesgue dominated convergence theorem, we attain imply that
$$
\int_{\Omega}|s_{\mu_{n}}v_{0}|^{q}
d\xi\rightarrow\int_{\Omega}
|s_{\mu}v_{0}|^{q}d\xi, \quad
\mbox{as} \quad n\rightarrow\infty.
$$
It now follows that
$$
\mu_{n}\int_{\Omega}|s_{\mu}v_{0}
|^{q}d\xi\rightarrow\infty, \quad
\mbox{as} \quad n\rightarrow\infty.
$$
Thus, invoking \eqref{4.4}, we can show that this cannot happen.
Therefore, $s_{\mu}\rightarrow0$ as $\mu\rightarrow\infty$.
Furthermore, we can apply  \eqref{4.4} to show that
$$
\lim_{\mu\rightarrow\infty}
\mu\int_{\Omega}|s_{\mu}v_{0}|^{q}d\xi=0
$$
and
$$
\lim_{\mu\rightarrow\infty}
\int_{\Omega}\int_{\Omega}
\frac{|s_{\mu}v_{0}(\xi)|^{Q_{\lambda}^{\ast}}
|s_{\mu}v_{0}(\eta)|^{Q_{\lambda}^{\ast}}}
{|\eta^{-1}\xi|^{\lambda}}d\eta d\xi=0.
$$
Therefore,
$s_{\mu}\rightarrow0$ as
$\mu\rightarrow\infty$ and by the definition of
$I_{\mu}$, we get that
$$
\lim_{\mu\rightarrow\infty}
(\sup_{s\geq0}I_{\mu}(sv_{0}))=
\lim_{\mu\rightarrow\infty}
I_{\mu}(s_{\mu}v_{0})=0 .
$$
So, there is
$\mu^{\ast}>0$ such that if
$\mu > \mu^{\ast}$,
 we have
$$
\sup_{s\geq0}I_{\mu}(sv_{0})
<(\frac{1}{2}
-\frac{1}{q})S_{HG}^{\frac{Q_{\lambda}^{\ast}}
{Q_{\lambda}^{\ast}-1}}.
$$
Letting $e=s_{1}v_{0}$ with $s_{1}$ large enough to have
$I_{\mu}(e)<0$, we get
$$
0<c_{\mu}\leq\max_{0\leq t\leq1}
I_{\mu}(h(t))\quad \text{where}
\quad h(t)=ts_{1}v_{0}.
$$
Therefore,
$$
0<c_{\mu}\leq\sup_{s\geq0}
I_{\mu}(sv_{0})<(\frac{1}{2}
-\frac{1}{q})S_{HG}^{\frac{Q_{\lambda}^{\ast}}
{Q_{\lambda}^{\ast}-1}}
$$
for
$\mu$ large enough.
The proof of
Theorem \ref{the1.2} is now complete.
\end{proof}

\section{Proof of Theorem 1.3 }
In order to prove Theorem \ref{the1.3}, we shall use the
Krasnoselskii genus theory \cite{Rabinowitz}. Let X be a Banach
space and let $\Lambda$ denote the family of all closed subsets
$A\subset X\backslash\{0\}$ which are symmetric with respect to the
origin, that is, $u\in A$ implies that also $-u\in A$. If
$z_{1},\cdots,z_{k}$ are in $Z$, then the span of all liner
combinations of $z_{1},\cdots,z_{k}$ is denoted by
$\mbox{span}\{z_{1},\cdots,z_{k}\}$ and is called the subset of $Z$
generated by $z_{1},\cdots,z_{k}$.

\begin{theorem}\label{the5.1}(see Rabinowitz \cite{Rabinowitz})
Let E be an infinite-dimensional Banach space and let $I\in
C^{1}(X)$ be even, with $I(0) = 0$. Assume that $E=X\oplus Y$, where
X is a finite-dimensional space, and I satisfies the following
properties:

$(i)$ There is $\theta>0$ such that $I$ satisfies the $(PS)_{c}$
condition for all $c\in(0,\theta)$.

$(ii)$
There are
$\rho,\alpha>0$ such that
$I(u)\geq \alpha$ for all
$u\in\partial B_{\rho}\bigcap Y$.

$(iii)$ For any finite-dimensional subspace $\tilde{E}\subset E$,
there is $R=R(\tilde{E})>\rho$ such that $I(u)\leq0$ on
$\tilde{E}\setminus B_{R}$.
\end{theorem}

Moreover, suppose that $X$ is $k$-dimensional and
$X=\mbox{span}\{z_{1},\cdots,z_{k}\}$. For $n\geq k$, inductively
select $z_{n+1}\notin X_{n}= \mbox{span}\{z_{1},\cdots,z_{n}\}$. Let
$R_n = R(X_n)$ and $\Upsilon_n=B_{R_{n}}\bigcap X_n$. Define
$$
W_n=\{\varphi\in C(\Upsilon_n,E):
\varphi|\partial_{B_{R_{n}}\bigcap X_n}=
id \ \mbox{and}\  \varphi
\ \mbox{is odd}\}
$$
and
$$
\Gamma_{i}=
\{\varphi(\overline{\Upsilon_n\setminus V}):
\varphi\in W_n,n\geq i,V\in\Lambda,\gamma(V)\leq n-i\},
$$
where
$\gamma(V)$ is the
 Krasnoselskii genus of
$V$. For
$i\in N$,
let
$$
c_i=\inf_{X\in\Gamma_{i}}
\max_{u\in X}I(u).
$$
Hence, $0\leq c_{i}\leq c_{i+1}$ and $c_{i}<\theta$ for $i>k$. Then
we get that $c_{i}$ is a critical value of $I$. Moreover, if
$c_{i}=c_{i+1}=\cdots=c_{i+p}=c<\theta$
 for
$i>k$, then
$\gamma(K_{c})\geq p+1$, where
$$
K_{c}=\{u\in E:I(u)=c \quad\mbox{and}
\quad I'(u)=0\} .
$$

\begin{proof}[\bf Proof of Theorem \ref{the1.3}]
It is well known that
$\mathring{S}_1^2(\Omega)$ is a reflexive
Banach space and
$I_{\mu}\in C^{1}(\mathring{S}_1^2(\Omega))$. The
energy functional
$I_{\mu}$ satisfies
$I_{\mu}(0)=0$. Our proof of Theorem  \ref{the1.3} will procede in 3 steps.

{\bf Step 1.}
 The argument is similar to the
proof of (i) and (ii) in Theorem
\ref{the4.1}. We can see that
$(ii)$ and
$(iii)$ of Theorem
\ref{the5.1} are satisfied.

{\bf Step 2.}
 We show that there is a sequence
$(\Psi_{n})_{n}\subset R^{+}$, with
$\Psi_{n}\leq \Psi_{n+1}$, satisfying
$$
c_{n}^{\mu}=\inf_{X\in\Gamma_{n}}
\max_{u\in X}I_{\mu}(u)<\Psi_{n}.
$$
For this purpose, applying an argument given in Wei and Wu
\cite{Wei}, according to the definition of $c_{n}^{\mu}$, one has
$$
c_{n}^{\mu}=\inf_{X\in\Gamma_{n}}
\max_{u\in X}I_{\mu}(u)\leq
\inf_{X\in\Gamma_{n}}
\max_{u\in X}\{\|u\|^{2}
-\frac{1}{2Q_{\lambda}^{\ast}}
\int_{\Omega}\int_{\Omega}
\frac{|u(\xi)|^{Q_{\lambda}^{\ast}}
|u(\eta)|^{Q_{\lambda}^{\ast}}}
{|\eta^{-1}\xi|^{\lambda}}d\eta
d\xi\}.
$$
Set
$$
\Psi_{n}=\inf_{X\in\Gamma_{n}}
\max_{u\in X}\{\|u\|^{2}
-\frac{1}{2Q_{\lambda}^{\ast}}
\int_{\Omega}\int_{\Omega}
\frac{|u(\xi)|^{Q_{\lambda}^{\ast}}
|u(\eta)|^{Q_{\lambda}^{\ast}}}
{|\eta^{-1}\xi|^{\lambda}}d\eta
d\xi\}.
$$
Then
$\Psi_{n}<\infty$ and
$\Psi_{n}\leq \Psi_{n+1}$,
 by the definition of
$\Gamma_{n}$.

{\bf Step 3.}
  We show that problem
  \eqref{1.1} has at least
$n$ pairs of
weak solutions.
 As in \eqref{4.3}, a similar
 discussion yields that there exists
$\mu^{\ast\ast}>0$ satisfying
$$
c_{n}^{\mu}\leq\Psi_{n}<(\frac{1}{2}
-\frac{1}{q})S_{HG}^{\frac{Q_{\lambda}^{\ast}}
{Q_{\lambda}^{\ast}-1}}\ \
\mbox{for\ all}\ \mu > \mu^{\ast\ast}.
$$
Thus, one has
$$0<c_{1}^{\mu}\leq c_{2}^{\mu}\leq\cdots\leq
c_{n}^{\mu}<\Psi_{n}<(\frac{1}{2}
-\frac{1}{q})S_{HG}^{\frac{Q_{\lambda}^{\ast}}
{Q_{\lambda}^{\ast}-1}}.$$ An application of Rabinowitz
\cite[Proposition 9.30]{Rabinowitz} ensures that the levels
$c_{1}^{\mu}\leq c_{2}^{\mu}\leq\cdots\leq
 c_{n}^{\mu}$
  are critical
values of
$I_{\mu}$.

If $c_{i}^{\mu}=c_{i+1}^{\mu}$ where $i = 1,2,\cdots,k-1$, then by
Ambrosetti and Rabinowitz \cite[Remark 2.12 and Theorem 4.2]{am},
the set $K_{c_{i}^{\mu}}$ consists of infinite number of different
points and problem \eqref{1.1} has infinite numbers of weak
solutions. Hence, problem \eqref{1.1} has at least $n$ pairs of
solutions. The proof of Theorem \ref{the1.3} is thus complete.
\end{proof}

\noindent {\bf Acknowledgements.} Song was supported by the National
Natural Science Foundation of China (No. 12001061), the Science and
Technology Development Plan Project of Jilin Province, China
(No.20230101287JC) and Innovation and Entrepreneurship Talent
Funding Project of Jilin Province (No.2023QN21). Repov\v{s} was
supported by the Slovenian Research Agency program No. P1-0292 and
grants Nos. J1-4031, J1-4001, N1-0278, N1-0114, and N1-0083.

\end{document}